\documentclass [11pt]{amsart}
\usepackage{amsmath}
\usepackage{amssymb}
\usepackage{comment}

\newtheorem{thm}{Theorem}[section]

\newtheorem{lem}[thm]{Lemma}

\theoremstyle{definition}

\theoremstyle{remark}
\newtheorem{rem}[thm]{Remark}
\numberwithin{equation}{section}

\newcommand{\beas}{\begin{eqnarray*}}
\newcommand{\eeas}{\end{eqnarray*}}
\newcommand{\bes} {\begin{equation*}}
\newcommand{\ees} {\end{equation*}}
\newcommand{\be} {\begin{equation}}
\newcommand{\ee} {\end{equation}}
\newcommand{\bea} {\begin{eqnarray}}
\newcommand{\eea} {\end{eqnarray}}
\newcommand{\ra} {\rightarrow}
\newcommand{\txt} {\textmd}
\newcommand{\ds} {\displaystyle}
\newcommand{\R}{\mathbb R}
\newcommand{\C}{\mathbb C}
\newcommand{\T}{\mathbb T}
\newcommand{\N}{\mathbb N}

\newcommand{\la}{\lambda}

\begin{document}

\title[Uncertainty Principles of Ingham and Paley-Wiener on Semisimple Lie Groups] {Uncertainty Principles of Ingham and Paley-Wiener on Semisimple Lie Groups}

\author{Mithun Bhowmik and Suparna Sen}

\address{Stat-Math Unit, Indian Statistical Institute, 203 B. T. Road, Kolkata - 700108, India.}

\email{mithunbhowmik123@gmail.com, suparna29@gmail.com}

\thanks{This work was supported by Indian Statistical Institute, India (Research fellowship to Mithun Bhowmik) and Department of Science and Technology, India (INSPIRE Faculty Award to Suparna Sen).}


\begin{abstract}
Classical results due to Ingham and Paley-Wiener characterize the existence of nonzero functions supported on certain subsets of the real line in terms of the pointwise decay of the Fourier transforms. Viewing these results as uncertainty principles for Fourier transforms, we prove certain analogues of these results on connected, noncompact, semisimple Lie groups with finite center. We also use these results to show unique continuation property of solutions to the initial value problem for time-dependent Schr\"odinger equations on Riemmanian symmetric spaces of noncompact type.
\end{abstract}

\subjclass[2010]{Primary 22E30; Secondary 22E46, 43A80}

\keywords{uncertainty principle, semisimple Lie group, symmetric space, Schr\"odinger equation}

\maketitle

\section{Introduction}
It is a well known fact in harmonic analysis that if the Fourier transform of an integrable function on $\R$ is very rapidly decreasing then the function can not be
compactly supported unless it vanishes identically. A manifestation of this fact is as follows: Let $f\in L^1(\R)$ and $a>0$ be such that 
\bes
|\widehat f(\xi)| \leq Ce^{-a|\xi|}, \:\:\:\: \txt{ for all }\xi \in \R.
\ees 
If $f$ is compactly supported then $f$ is identically zero. This holds due to the fact that the very rapid decay of the Fourier transform imposes real analyticity on the function. In fact, $f$ extends to a holomorphic function on an open subset of $\C$. 
This initial observation motivates one to endeavour for a more optimal decay of the Fourier transform $\widehat{f}$ for such a conclusion. For instance we may ask:  if $\widehat{f}$ decays faster than $1/(1+|\cdot|)^n$ for all $n \in \N$ but slower than the function  $e^{-a|.|}$, can $f$ be compactly supported without being identically zero? A more precise question  could be:  is there a nonzero integrable compactly supported function $f$ on $\R$ with its Fourier transform satisfying 
\be \label{introdecay} 
|\widehat f(\xi)| \leq Ce^{-\frac{|\xi|}{\log|\xi|}}, \:\:\:\: \txt{ for large } |\xi|?
\ee
The answer to the above question is in the negative and follows from classical results due to Paley-Wiener (\cite{PW}, Theorem II; \cite{PW1}, P. 16, Theorem XII) and Ingham \cite{I} (see also Theorem \ref{euclidean}). These results may be viewed as instances of uncertainty principles in harmonic analysis.  We refer the reader to \cite{FS, HJ, T} for the literature on uncertainty principles.

Despite being results of same genre the treatment of Paley-Wiener and Ingham mentioned above have intriguing differences: while Paley-Wiener's method is complex analytic, Ingham relies on the notion of quasianalytic functions and the celebrated Denjoy-Carleman theorem (\cite{R}, Theorem 19.11). We would like to point out that the result of Paley and Wiener leads to a proof of the Denjoy-Carleman theorem. In some sense the result of Ingham is stronger than that of Paley and Wiener. In fact, a close examination of the proof of Ingham's result (as given in \cite{I}) reveals that if $\widehat f$ satisfies (\ref{introdecay}) and $f$ vanishes on an open set then $f$ actually vanishes identically. Different versions of the result of Ingham was later proved by Levinson and Beurling \cite{L1, L2, Koo}.
All these results mentioned above are obtained only for the circle group and real line. Only very recently we have obtained the following analogues of the results of Paley-Wiener and Ingham in the context of $n$-dimensional Euclidean spaces (see Theorem 2.3 and Theorem 2.2 of \cite{BRS} respectively). For $f\in L^1(\R^n)$, we shall define its Fourier transform $\widehat f$ by
\bes
\widehat f(\xi)=\int_{\R^n}f(x)~ e^{-i x\cdot \xi}~ dx,\:\:\:\: \txt{ for } \xi\in\R^n.
\ees
\begin{thm} \label{euclidean}
Let $\psi: [0,\infty) \rightarrow[0,\infty)$ be a locally integrable function and $f \in L^1(\R^n)$ be such that 
\be \label{estintro}
|\widehat{f}(\xi)|\leq Ce^{-\psi(\|\xi\|)},  \:\:\:\: \txt{ for all } \xi \in\R^n.
\ee
Let
\be \label{integral}
I = \int_{0}^{\infty}{\frac{\psi(r)}{1+r^2}dr}.
\ee
\begin{enumerate}
\item[(a)] If $f\in C_c(\R^n)$ and $I=\infty$ then $f$ is identically zero on $\R^n$. Conversely, if $I< \infty$ and $\psi$ is non-decreasing then there exists a nonzero function $f \in C_c^{\infty}(\R^n)$ satisfying (\ref{estintro}).
\item[(b)] Let $\psi(r) = r \theta(r)$, for $r \in [0,\infty)$, where $\theta: [0,\infty)\rightarrow [0,\infty)$ is a decreasing function with $\lim_{r \to\infty}\theta(r)=0$. If $f$ vanishes on an open set in $\R^n$ and $I = \infty$ then $f$ is zero. Conversely, if $I<\infty$ then there exists a nonzero function $f \in C_c^{\infty}(\R^n)$ satisfying (\ref{estintro}).
\end{enumerate}
\end{thm}
\noindent Part (a) of the above theorem corresponds to the result of Paley-Wiener and part (b) corresponds to that of Ingham.

Viewing $\R^n$ as a noncompact Riemannian symmetric space one can naturally ask about analogues of the above results for Riemannian symmetric spaces of noncompact type or more generally for connected, noncompact, semisimple Lie groups with finite centre. In these cases nice parametrization of relevant representations are available and consequently the natural domain of the Fourier transform turns out to be $\R$ or $\R^n$. This enables us to formulate analogous questions for these spaces. In this paper we will first prove an analogue of the result of Paley-Wiener for connected, noncompact semisimple Lie groups with finite centre and arbitrary rank (Theorem \ref{semisimple}). The main ingredient here is Lemma \ref{lemma} whose proof was inspired by a related result of Hirschmann \cite{Hi}.

However, an analogue of Ingham's result poses much greater difficulty. This is due to the unavailability of a suitable analogue of the Denjoy-Carleman theorem for Riemannian manifolds. Certain analogues of the Denjoy-Carleman theorem for Riemannian manifolds were proved by Bochner and Taylor \cite{Bo, BT} but these results don't seem to be suitable for our purpose. We therefore restrict ourselves to bi-$K$-invariant functions on complex semisimple Lie groups and obtain an analogue of Ingham's theorem (Theorem \ref{complexsemisimple}). The main idea here is to use the explicit description of elementary spherical functions and reduce matters to Euclidean spaces.

Our next set of results deal with Riemannian symmetric spaces of noncompact type. Our aim here is to relate the results of Paley-Wiener and Ingham to the problem of unique continuation of solution to the Schr\"odinger equation in the spirit of \cite{EKPV1, EKPV, C, PS}. These results present some pleasant surprise in the sense that the expected results demand more decay than that is required in Theorem \ref{semisimple} and Theorem \ref{complexsemisimple} (see the beginning of \S 4.2.1 and Remark \ref{remark}). This can be attributed to the exponential volume growth of the invariant measure on the symmetric space. Taking this into account we could obtain unique continuation properties of the solution to the Schr\"odinger equation on symmetric spaces (Theorem \ref{schrcomplexthm}, Theorem \ref{schrthm}).

This paper is organized as follows: In the next section we describe the required preliminaries on connected, noncompact, real semisimple Lie groups with finite centre and the associated symmetric spaces. In Section 3 we shall prove an analogue of Theorem \ref{euclidean} (a) for $f \in C_c(G)$ where $G$ is a connected, noncompact, real semisimple Lie group with finite centre. We also prove an analogue of Theorem \ref{euclidean} (b) on a noncompact, complex semisimple Lie group. In the last section we consider unique continuation properties of solutions to the initial value problem for time-dependent Schr\"odinger equation on Riemannian symmetric spaces of the noncompact type. First we prove such a result for bi-$K$-invariant initial value on noncompact, complex semisimple Lie groups related to Theorem \ref{euclidean} (b). We also prove such a result on Riemannian symmetric spaces of the noncompact type related to Theorem \ref{euclidean} (a).

\section{Notation and Preliminaries}

In this section, we shall discuss the preliminaries and notation related to the noncompact semisimple Lie groups and the associated symmetric spaces. These are standard and can be found, for example, in \cite{GV, H, H1, H2, K}. To make the article self-contained, we shall gather only those results which will be used throughout this paper. 

We shall use the following notation in this paper: $C_c(X)$ denotes the set of compactly supported continuous functions on $X$, $C_c^{\infty}(X)$ denotes the set of compactly supported smooth functions on $X$ and $C$ denotes a constant whose value may vary. For $x,y \in \R^n$, we shall use $\|x\|$ to denote the Euclidean norm of the vector $x$ and $x \cdot y$ to denote the Euclidean inner product of the vectors $x$ and $y$. We shall also denote the Euclidean norm in $\C^n$ by $\|\cdot\|$. $\Im (\la)$ denotes the imaginary part of $\la \in \C^n$. 

Let $G$ be a connected, noncompact, real semisimple Lie group with finite centre and $K$ be a fixed maximal compact subgroup of $G$. Let $\mathfrak g$ and $\mathfrak k$ denote the Lie algebras of $G$ and $K$ respectively. Suppose that $B$ is the Cartan Killing form on $\mathfrak g$ and $\mathfrak g = \mathfrak k \oplus \mathfrak p$ is a Cartan decomposition of $\mathfrak g$. It is known that $B$ restricted to $\mathfrak p$ is positive definite, thus it gives an inner product and hence a norm $\| \cdot \|_B$ on $\mathfrak p$. Fix a maximal abelian subspace $\mathfrak a$ of $\mathfrak p$. If the dimension of $\mathfrak a$ is $l$, then $G$ is said to be of real rank $l$. We can identify $\mathfrak a$ with $\mathbb{R}^l$ endowed with the inner product induced from $\mathfrak p$. Let $\mathcal{R}$ denote the set of nonzero roots for the adjoint action of $\mathfrak a$ on $\mathfrak g$ and $W$ denote the Weyl group corresponding to $\mathcal{R}$. Fix a Weyl chamber $\mathfrak a_+$ of $\mathfrak a$ and let $\mathcal{R}_+$ be the corresponding set of positive roots. 

Let $A=\exp {\mathfrak a}$ and $A_+= \exp{\mathfrak{a_+}}$. If $\overline{A_+}$ denotes the closure of $A_+$ in $G$, then one gets the polar decomposition $G=K\overline{A_+}K$, that is, each $g\in G$ can be uniquely written as 
\bes
g=k_1 a k_2, \:\:\:\: \txt{ for } k_1, k_2 \in K \txt{ and } a \in {\overline{A_+}}.
\ees 
Using the polar decomposition $\|\cdot\|_B$ on $G$ is defined by
\bes
\|g\|_B = \|k_1 a k_2\|_B = \|\log a\|_B,
\ees 
where $\log a$ denotes the unique element in $\mathfrak a$ such that $\exp(\log a)=a$. Let $\mathfrak g_\alpha$ denote the root space corresponding to $\alpha\in \mathcal{R}$ with $m_{\alpha}=dim ~\mathfrak g_{\alpha}$. The Haar measure $dg$ on $G$ relative to the polar decomposition is given by $dg = J(a)~dk_1~da~dk_2$, that is, for any suitable function $f$ on $G$, we have
\bes
\int_{G}{f(g)dg}=\int_{K}{\int_{\overline{A_+}}{\int_{K}{f(k_1ak_2) ~ J(a)~dk_1~da~dk_2}}},
\ees  
where $dk$ is the normalised Haar measure on $K$, $da$ is the Lebesgue measure on $A \cong \R^l$ and
\bes 
J(a)= \prod_{\alpha\in \mathcal{R}_+}\left(e^{\alpha(\log a)}-e^{-\alpha(\log a)}\right)^{m_{\alpha}}, \:\:\:\: \txt{ for } a \in \overline{A_+}.
\ees
We define the half-sum of the elements of $\mathcal{R}_+$ counted with their multiplicities by $\rho$ given by 
\bes
\rho=\frac{1}{2}\sum_{\alpha\in \mathcal{R}_+}m_{\alpha}\alpha.
\ees
Thus we get the trivial estimate
\be \label{jest} 
J(a) \leq C e^{2\rho(\log a)}, \:\:\:\: \txt{ for } a \in \overline{A_+}.
\ee 

For an element $\lambda$ in the dual $\mathfrak a^*$ of $\mathfrak a$, let $H_\lambda$ be the unique element in $\mathfrak a$ such that 
\bes 
\lambda(H)= B(H, H_\lambda), \:\:\:\: \txt{ for all } H \in \mathfrak a.
\ees 
The map $\la \mapsto H_{\la}$ identifies $\mathfrak a^*$ with $\mathfrak a$, and we use it to define the dual inner product on $\mathfrak a^*$, also denoted $B(\cdot,\cdot)$, by the formula 
\bes 
B(\lambda, \mu) = B(H_\lambda, H_\mu), \:\:\:\: \txt{ for }\lambda, \mu \in \mathfrak a^*.
\ees 
The elements of the Weyl group $W$ are orthogonal transformations of $\mathfrak a^*$ which correspond to orthogonal transformations of $\mathfrak a$ by the formula 
\be \label{shlambda}
s H_{\la} = H_{s \la}, \:\:\:\: \txt{ for } s \in W \txt{ and } \la \in \mathfrak a^*.
\ee 
The bilinear extension of $B(\cdot,\cdot)$ to $\mathfrak a^*_{\mathbb{C}}$, the dual of the complexification $\mathfrak a_{\C}$ of $\mathfrak a$, is also denoted by $B(\cdot,\cdot)$. 

We now describe the principal series representations of $G$. The Iwasawa decomposition $G=KAN$ gives rise to the projection mappings $\kappa:G\rightarrow K$, $H : G \ra \mathfrak{a}$ and $\eta:G\rightarrow N$ such that 
\bes
g=\kappa(g)\exp H(g)\eta(g).
\ees 
Let $M$ be the centraliser of $A$ in $K$. Given $\xi$ in the unitary dual $\widehat M$ of $M$ let $\mathbb{H}$ be the finite dimensional Hilbert space on which $\xi$ is realised. We define the Hilbert space $\mathbb{H}_{\xi}$ by  
\bes 
\left\{ \phi : K \ra \mathbb{H} \txt{ measurable, } \phi(km) = \xi(m^{-1}) \phi(k) \txt{ for } k \in K, ~ m \in M \txt{ and } \int_K \|\phi(k)\|^2 dk < \infty \right\},
\ees 
where $\|\cdot\|$ denotes the norm on $\mathbb H$ induced from the inner product $\langle \cdot, \cdot \rangle$ on $\mathbb{H}$. The Hilbert space $\mathbb{H}_{\xi}$ is equipped with the inner product
\bes
\langle \phi , \psi \rangle_{\xi} = \int_K \langle \phi(k), \psi(k) \rangle ~ dk.
\ees
For $\xi \in \widehat M $ and $\lambda \in \mathfrak a^*_{\mathbb{C}}$ we have a representation $\pi_{\xi, \lambda}$ acting on the Hilbert space $\mathbb{H}_{\xi}$ given by  
\bes
\left(\pi_{\xi, \lambda}(g)\phi\right)(k)=e^{(i\lambda-\rho)H(g^{-1}k)}\phi\left(\kappa(g^{-1}k)\right), \:\:\:\: \txt{ for } g \in G, ~ k \in K, ~ \phi \in \mathbb H_{\xi}.
\ees 
It is known that $\pi_{\xi, \lambda}$ is unitary if $\lambda \in \mathfrak a^*$ and $\pi_{\xi, \lambda}$, $\pi_{\xi, \mu}$ are unitarily equivalent if and only if $\la = s \mu$ where $s \in W$. Moreover, given $\xi\in \widehat M$ there exists a dense open subset $O_\xi\subset \mathfrak a^*$ such that for $\lambda \in O_\xi$, $\pi_{\xi, \lambda}$ is irreducible. We now define the group Fourier transform of $f \in L^1(G)$ given by the operator valued integral 
\bes 
\pi_{\xi, \lambda}(f) = \int_{G} f(g) \pi_{\xi, \lambda}(g) dg, \:\:\:\: \txt{ for } (\xi, \lambda) \in \widehat{M} \times \mathfrak a^*.
\ees 
For $f \in L^1 \cap L^2(G),$ $\pi_{\xi, \lambda}(f)$ is a Hilbert-Schmidt operator and we shall denote its Hilbert-Schmidt norm by $\|\pi_{\xi, \lambda}(f)\|_{HS}$.

If $\xi$ is the trivial representation in $\widehat M$, then we denote $\pi_{\xi,\lambda}$ by $\pi_\lambda$ and the set of representations $\{\pi_\lambda\}_{\lambda\in \mathfrak{a}^*}$, realized on the Hilbert space $L^2(K/M)$, are called the class one principal series representations of $G$. We observe that $\pi_{\la}|_K$ are  given by left translations on $L^2(K/M)$, in particular, the $K$-fixed vectors are given by constant functions. The right-$K$-invariant functions on $G$ can be viewed as functions on the symmetric space $X=G/K$ and vice-versa. For the harmonic analysis of a function $f$ on $X$, only the class one principal series representations $\pi_\lambda$ are relevant and $\pi_{\lambda}(f)$ is completely determined by $\pi_{\lambda}(f)e_0$ where $e_0$ denotes the constant function $1$ on $K/M$.  In this case, the group theoretic Fourier transform can be reinterpreted as the Helgason-Fourier transform on the symmetric space $X$, as introduced by Helgason. For a sufficiently nice function $f$ on $X$, its Helgason-Fourier transform $\widetilde{f}$ is a function defined on $\mathfrak{a}_{\C}^* \times K/M$ given by 
\be \label{hftdefn}
(\pi_{\lambda}(f)e_0)(kM) = \widetilde{f}(\lambda,kM) = \int_{G} f(g) e^{(i\lambda - \rho)H(g^{-1}k)} dg, 
\ee
for  $\lambda \in \mathfrak{a}_{\C}^*, ~ kM \in K/M$ whenever this integral exists (see \cite{H1}, Ch. III, \S 1). Thus it follows that 
\be \label{hftftrel}
\int_{K/M} |\widetilde{f}(\lambda,kM)|^2 dk = \|\pi_{\la}(f) e_0\|_{L^2(K/M)}^2 = \| \pi_{\la}(f) \|_{HS}^2, \:\:\:\: \txt{ for } \la \in \mathfrak{a}^*.
\ee
The Plancherel theorem states that the Helgason Fourier transform extends to an isometry of $L^2(X)$ onto $L^2(\mathfrak{a}_+^* \times K/M, |c(\lambda)|^{-2} dk ~ d\lambda)$ where $c(\lambda)$ is Harish-Chandra's $c$-function, that is,
\bes
\int_X |f(x)|^2 dx = \int_{\mathfrak{a}_+^*} \int_{K/M} |\widetilde{f}(\lambda,kM)|^2 |c(\lambda)|^{-2} dk ~ d\lambda.
\ees

Using the polar decomposition of $G$ we may view a bi-$K$-invariant function $f$ on $G$ as that depending only on its values on $A_+$, or by using the inverse exponential map we may also view $f$ as a function on $\mathfrak{a}$ solely determined by its values on $\mathfrak{a}_+$. Henceforth we shall denote the set of bi-$K$-invariant functions in $L^1(G)$ by $L^1(K \backslash G/K)$. If $f \in L^1(K \backslash G/K)$ then $\pi_{\lambda}(f)$ is determined by $\langle \pi_{\lambda}(f)e_0, e_0 \rangle$ and the spherical transform $\widetilde{f}(\lambda)$ of $f$ is defined by 
\bes
\widetilde{f}(\lambda) = \int_G f(g) \phi_{\lambda}(g)dg = \langle \pi_{\lambda}(f)e_0, e_0 \rangle, \:\:\:\: \txt{ for } \la \in \mathfrak{a}^*,
\ees
where $\phi_\lambda$ is the elementary spherical function corresponding to $\lambda\in \mathfrak a^*$ given by 
\be \label{philambda} 
\phi_\lambda(g)=\langle \pi_\lambda(g) e_0, e_0 \rangle = \int_K e^{(i\la - \rho)H(g^{-1}k)} dk, \:\:\:\: \txt{ for } g \in G.
\ee
It is clear that the elementary spherical function $\phi_{\lambda}$ is a bi-$K$-invariant function on $G$ for any $\lambda\in \mathfrak a^*$. The following estimates regarding $\phi_\la$ are well known (see \cite{GV}, \S 4.6).
\be
e^{-\rho(\log a)} \leq \phi_0(a) ~~\leq ~ C (1+\|\log a\|_B)^m e^{-\rho(\log a)}, \:\:\:\: \txt{ for } a \in \overline{A_+}, \txt{ some } m,C>0, \label{phi0}
\ee
\be
0 < \phi_{i\la}(a) \leq ~ e^{\la^+(\log a)} \phi_0(a), \:\:\:\: \txt{ for } a \in \overline{A_+},~ \lambda\in \mathfrak a^*,  \label{phiila} 
\ee
where $\la^+$ is the element in the fundamental Weyl chamber corresponding to $\la$. In this case
\be \label{sphtftrel}
\|\pi_{\lambda}(f)\|_{HS} = |\langle \pi_{\lambda}(f)e_0, e_0\rangle| = |\widetilde{f}(\lambda)|, \:\:\:\: \txt{ for } \la \in \mathfrak{a}^*.
\ee

Apart from Fourier transform, we will also need the notion of Radon transform on Riemannian symmetric space $X$. In the following, we shall describe the properties of Radon transform which will be used in this paper. By identifying the space $G/MN$ of horocycles on $X$ with $ (K/M) \times A$, we define the Radon transform of a sufficiently regular function $f:X\rightarrow \mathbb{C}$  by 
\be \label{radon}
Rf(kM,a)=e^{\rho(\log a)}\int_N{f(kan\cdot o)~ dn}
\ee
for all $kM \in K/M$, $a\in A$ whenever this integral exists (see \cite{H1}, P. 220, see also \cite{SSG}) and
$o = eK$ denotes the identity coset in $X$. It is known that the Helgason-Fourier transform of a sufficiently regular function $f$ is the Euclidean Fourier transform of the Radon transform (see \cite{H1}, P. 219) given by
\be \label{hftradon}
\widetilde{f}(\lambda,kM) = \mathcal{F}_A(Rf(kM,\cdot))(\lambda), \:\:\:\: \txt{ for almost all } kM \in K/M \txt{ and all } \lambda \in \mathfrak{a}^*,
\ee 
where $\mathcal{F}_A$ denotes the Euclidean Fourier transform of $h \in L^1(A)$ defined by
\bes
(\mathcal{F}_A h)(\lambda) = \int_A h(a) e^{-i\lambda(\log a)} da, \:\:\:\: \txt{ for } \lambda \in \mathfrak{a}^* \cong \R^l.
\ees
It is well known that Radon transform is injective on $L^1(X)$ (see \cite{H1}, Ch. II, Theorem 3.2). 

For $f \in L^1(K \backslash G/K)$, we define the Abel transform by the following integral 
\bes
\mathcal{A}f(a)=e^{\rho(\log a)}\int_N{f(an \cdot o)~ dn}, \:\:\:\: \txt{ for } a \in A,
\ees
(see \cite{H1}, P. 381). Note that the Radon transform of a bi-$K$-invariant function reduces to the Abel transform. It follows from (\ref{hftradon}) that the spherical transform of $f \in L^1(K \backslash G/K)$ is the Euclidean Fourier transform of the Abel transform $\mathcal{A}f$ given by
\be \label{abelsphtrel}
\widetilde{f}(\lambda) = \mathcal{F}_A({\mathcal{A}f})(\lambda), \:\:\:\: \txt{ for } \lambda \in \mathfrak{a}^* \cong \R^l.
\ee
This property of Abel transform is crucial for reducing some questions on bi-$K$-invariant functions defined on a semisimple Lie group $G$ to related questions on $\R^l$. Moreover, it is known that the Abel transform $\mathcal{A}$ induces a bijection between $C_c^{\infty}(K \backslash G/K)$ and the set of $W$-invariant functions in $C_c^{\infty}(A)$ (see \cite{H1}, Ch. IV, Theorem 4.1).

\section{Uncertainty Principles}

In this section we shall present analogues of the results of Paley-Wiener and Ingham, namely Theorem \ref{euclidean} (a) and Theorem \ref{euclidean} (b) on connected, noncompact, real semisimple Lie groups with finite centre and noncompact, complex semisimple Lie groups respectively. 

\subsection{Real Semisimple Lie Groups}

To prove the first theorem we shall prove a lemma on entire functions on $\C^n$. The proof of this lemma for the one-variable case is given in \cite{BS}. Since the paper is yet to be accepted, we reproduce the proof here for the sake of completeness. To prove this lemma, we shall need two results. The first one is regarding upper semicontinuous functions.
\begin{thm} [\cite{Co}, P. 218, Theorem 3.6] \label{uppsemicont}
 Let $(X,d)$ be a metric space, $v: X\rightarrow [-\infty, \infty)$ be upper semi-continuous and $v\leq M< \infty$ on $X$. Then there exists a decreasing sequence of uniformly continuous functions $\{f_n\}$ on $X$ such that $f_n\leq M$ and for every $x\in X$, $f_n(x)$ decreases to $v(x)$.
\end{thm}
\noindent The second result we need is an analogue of the maximum modulus principle on unbounded domain for subharmonic functions. We briefly recall the definition of subharmonic functions. Let $D$ be an open subset of $\C$. A function $u : D \ra [-\infty, \infty)$ is called subharmonic if $u$ is upper semi-continuous and satisfies the local submean inequality, that is,  for any $w \in D$ there exists $\rho > 0$ such that for all $r\in (0, \rho)$ the following holds
\begin{equation*} 
u(w) \leq \frac{1}{2\pi} \int_0^{2\pi} u(w + r e^{it}) dt .
\end{equation*}
It is well known that if $f$ is a holomorphic function then $g(z)=\log (|f(z)|)$ is a subharmonic function (see \cite{R}, P. 336). 
\begin{thm} [\cite{B}, P. 224, Theorem 7.15] \label{maxmod}
 Let $\Omega$ be a region (not necessarily bounded) and $u: \Omega \rightarrow \R$ be a subharmonic function which is bounded above. Let $A$ be a proper, countable subset of the boundary $\partial \Omega$ of $\Omega$ and $M$ a finite constant such that $\ds{\varlimsup_{z \to \xi} u(z)\leq M}$ for all $\xi \in \partial \Omega\smallsetminus  A$. Then $u \leq M$ throughout $\Omega$.
\end{thm}

We shall now present the required lemma on entire functions. In the following we shall interpret a radial function on $\R^n$ as an even function on $\R$.

\begin{lem}\label{lemma}
Let $f$ be an entire function on $\mathbb{C}^n$ and $\psi$ be a non-negative, radial, locally integrable function on 
$\mathbb{R}^n$ such that for positive constants $a$ and $C,$ we have 
\bea
|f(z)| &\leq& Ce^{a\|z\|}, \:\:\:\: \txt{ for all } z\in \mathbb{C}^n, \label{fCdecay}  \\ 
|f(x)| &\leq& Ce^{-\psi(x)}, \:\:\:\: \txt{ for all } x\in \mathbb{R}^n. \label{fRdecay}
\eea
If 
\bes 
\int_{\mathbb{R}}{\frac{\psi(r)}{1+r^2}dr}=\infty,
\ees 
then $f$ is identically zero on $\mathbb{C}^n.$
\end{lem}

\begin{proof}

First we shall prove the result for $n=1$. We consider the function 
\begin{equation*}
g(z)=\frac{1}{C}e^{iaz}f(z),\:\:\:\:\text{for all $z\in\C$,}
\end{equation*}
and observe that $g$ is an entire function. We want to apply Phragm\'{e}n-Lindel\"of theorem to show that for all $z$ in the closed upper half plane $\overline{\mathbb H}=\{z\in \C: \Im z \geq 0\}$ 
\begin{equation}
|g(z)|\leq 1,\label{boundone}
\end{equation}
(see \cite{SS}, P. 124, Theorem 3.4). Let $Q_1 = \{z = x + iy \in \mathbb C: x > 0, y>0\}$. 
It follows from the estimate (\ref{fCdecay}) that
\begin{equation*}
|g(iy)|=\frac{1}{C}e^{-ay}|f(iy)|\leq e^{-ay}e^{ay}=1,\:\:\:\:\text{for all $y>0$.}
\end{equation*}
It is also immediate from (\ref{fRdecay}) that 
\begin{equation*}
|g(x)|=\frac{1}{C}|f(x)| \leq e^{-\psi(x)}\leq 1, \:\:\:\: \txt{ for } x\in\R.
\end{equation*}
In particular, $g$ is bounded by $1$ on the positive real and positive imaginary axes. As $g$ satisfies the estimate (\ref{fCdecay}) we can apply the Phragm\'{e}n-Lindel\"of theorem to the sector $Q_1$ to obtain (\ref{boundone}). A similar argument for the quadrant $Q_2 = \{z = x + iy \in \mathbb C: x < 0, y>0\}$ proves the estimate (\ref{boundone}) for all $z \in \overline{\mathbb H}$. Since $g$ is an entire function, $\log |g|$ is a subharmonic function on $\C$ and 
\be \label{bound} 
\log |g(z)| \leq 0,\:\:\:\:\text{for all $z \in  \overline{\mathbb H}$.} 
\ee  
Now, we apply Theorem \ref{uppsemicont} for $X=\overline{\mathbb H}, v=\log |g|$, and $M=0$. Then there exists a decreasing sequence of uniformly continuous functions ${f_n}$ on $\overline{\mathbb H}$ such that $f_n \leq 0$ and $f_n(z)$ decreases to $v(z)$ for every $z\in \overline{\mathbb H }$. We define 
\begin{equation*}
v_n(z)= \max\{f_n(z), -n\},\:\:\:\: \txt{ for } z\in \overline{\mathbb H},\:\:n \in \mathbb N.
\end{equation*} 
It is clear that $\{v_n\}$ is a decreasing sequence of continuous functions. As $f_n$ takes only negative values it follows that $v_n(z)\in [-n,0]$ for all $z\in\overline{\mathbb H}$. In particular, $v_n$ is bounded for each $n\in\mathbb N$. We now claim that $\{v_n\}$ converges pointwise to $v$ on $\overline{\mathbb H}$. We first assume that $v(z)=-\infty$. Then $\{f_n(z)\}$ and $\{-n\}$ both converge to $-\infty$ and hence so does $\{v_n(z)\}$. Now assume that $v(z)$ is finite. In this case $f_n(z)\in (v(z)-1, v(z)+1)$ for all large $n$ and hence $v_n(z)=f_n(z)$ for all large $n\in\mathbb N$. It follows that $\{v_n(z)\}$ converges to $v(z)$ for all $z\in\overline{\mathbb H}$.
Let $U_n$ be the Poisson integral of the restriction of the function $v_n$ on $\R$ given by
\be \label{U_n}
U_n(x+iy)=\frac{1}{\pi}\displaystyle\int_{\R}{\frac{yv_n(t)}{y^2+(x-t)^2}dt},\:\:\:\: \txt{ for } x\in\R,\:y>0.
\ee
Since $v_n\in L^\infty(\R)$, the above integral exists and defines a harmonic function on the open upper half plane $\mathbb H=\{z\in \C: \Im z>0\}$. Moreover, since $v_n$ is continuous, we can extend $U_n$ to $\overline{\mathbb H}$ as a continuous function by letting $U_n(x) = v_n(x)$ for $x \in \R$ (\cite{SW}, P. 47, Theorem 2.1(b)).
We now define 
\begin{equation} \label{V_n}
V_n(z) = \log|g(z)|- U_n(z),\:\:\:\: \txt{ for } z\in\mathbb H.
\end{equation}
Since $U_n$ is harmonic, we get that $V_n$ is subharmonic on $\mathbb H$. Since
\begin{equation*}
v_n(t) \geq -n,\:\:\:\:\text{for all $t\in\R$,}
\end{equation*}
it follows from the definition of $U_n$ given in (\ref{U_n}) that 
\begin{equation*}
U_n(z)\geq -n,\:\:\:\: \txt{ for } z\in\mathbb H.
\end{equation*}
Using (\ref{bound}) we get that 
\begin{equation*}
V_n(z) \leq n,\:\:\:\:\text{for all $z\in \mathbb H$.}
\end{equation*}
In particular, $V_n$ is bounded above for each $n\in\mathbb N$. Since 
\begin{equation*}
v(z) = \log|g(z)| \leq v_n(z),\:\:\:\:\text{for all $z \in \overline{\mathbb H}$,}
\end{equation*}
it follows that
\begin{equation*}
\lim_{y \to 0} V_n(x+iy)=\log |g(x)|-v_n(x) \leq 0,\:\:\:\:\text{for all $x\in\R$.}
\end{equation*}
We now apply Theorem \ref{maxmod} for $\Omega= \mathbb H$, $u= V_n$ and $A= \phi$, the empty set to conclude that
\begin{equation*}
V_n(z) \leq 0,\:\:\:\:\text{for all $z\in \mathbb H$.}
\end{equation*}
It follows from (\ref{V_n}) that 
\begin{equation*}
\log|g(x+iy)|\leq \frac{1}{\pi}\int_{\R}{\frac{yv_n(t)}{y^2+(x-t)^2}dt}, \:\:\:\:\text{for all $y>0$, $x\in\R$}.
\end{equation*} 
Since $\{v_n\}$ is a decreasing sequence, by using monotone convergence theorem and taking limit as $n\rightarrow \infty$ in the inequality above we get \bes
\log|g(x+iy)| \leq \frac{1}{\pi}\displaystyle\int_{\R}{\frac{y\log|g(t)|}{y^2+(x-t)^2}dt}.
\ees 
The estimate (\ref{fRdecay}) now implies that
\begin{equation*}
\log|g(x+iy)| \leq -\frac{1}{\pi}\int_{\R}{\frac{y \psi(t)}{y^2+(x-t)^2}dt} \leq -C_{x,y}\int_{\R}{\frac{\psi(t)}{1+t^2}dt} = - \infty,
\end{equation*}
where $C_{x,y}$ is a positive constant which depends on $x$ and $y$.
So, for each $x\in \R$ and $y$ positive, it follows that $g(x+iy)=0$. As $f$ is an entire function it follows that $f(z)=0$ for all $z\in \C$. This proves the lemma for $n=1$.

Now, we shall prove the case $n>1$. Let $\xi \in \mathbb{R}^{n-1}$ be fixed and we define 
\bes 
g(z)=f(\xi,z), \:\:\:\: \txt{ for } z\in \mathbb{C},
\ees 
which is clearly an entire function on $\mathbb{C}$. Using (\ref{fCdecay}) we have for any $z \in \mathbb{C}$ 
\bes 
|g(z)|=|f(\xi,z)|\leq Ce^{a\|(\xi,z)\|} = Ce^{a\|(\xi,0)+(0,z)\|}\leq C_{\xi} e^{a|z|}.
\ees
Using (\ref{fRdecay}) we get that for any $x \in \mathbb{R}$ 
\bes
|g(x)|=|f(\xi,x)|\leq e^{-\psi(\xi,x)}=e^{-\psi_{\xi}(x)},
\ees 
where the function $\psi_{\xi}$ on $\R$ is defined as 
\bes
\psi_{\xi}(x)=\psi(\xi,x), \:\:\:\: \txt{ for } x \in \R.
\ees 
Since $\psi$ is radial, $\psi_{\xi}$ is an even function on $\R$. 
Moreover using the radiality of $\psi$ and the change of variable $\|\xi\|^2 + x^2 = r^2$ we have 
\beas
\int_{\mathbb{R}}{\frac{\psi_{\xi}(x)}{1+x^2}dx} 
&=& \displaystyle\int_{\mathbb{R}}{\frac{\psi(\sqrt{\|\xi\|^2 + x^2})}{1+x^2}dx} \\
&=& \displaystyle 2 \int_{\|\xi\|}^{\infty}{\frac{\psi(r)}{1+r^2-\|\xi\|^2} \frac{r}{\sqrt{r^2-\|\xi\|^2}}dr} \\
&\geq& 2 \displaystyle\int_{\|\xi\|}^{\infty}{\frac{\psi(r)}{1+r^2}dr} \\
&=& \infty.
\eeas 
Applying the lemma for the case $n=1$ on $g,$ we get that $g$ is zero on $\mathbb{C}$. So it follows that $f(\xi,z)$ is zero for all $z\in \mathbb{C}$. Since this is true for all $\xi \in \mathbb{R}^{n-1}$ we obtain that $f$ is identically 
zero on $\mathbb{R}^n$. Since $f$ is an entire function on $\mathbb{C}^n$ which vanishes on $\mathbb{R}^n$, we can conclude that $f$ vanishes on $\mathbb{C}^n$.

\end{proof}

We shall need few other facts to prove our first main theorem. Let $G$ be a connected, noncompact, semisimple Lie group with finite centre. We shall continue to assume the notation introduced in the previous section. For $\xi \in \widehat M$ let $\{e^{\xi}_j : j \in \mathbb{N}\}$ be an orthonormal basis of $\mathbb{H}_{\xi}$ consisting of $K$-finite vectors. For $g \in G,$ $\lambda \in \mathfrak a ^*$ and $m, n \in \mathbb{N}$ we define
\be \label{phimn}
\Phi^{m,n}_{\xi, \lambda}(g)= \left\langle  \pi_{\xi,\lambda}(g)e^\xi_m, e^\xi_n \right\rangle _\xi = \int_K e^{(i\lambda-\rho)H(g^{-1}k)} \left\langle e^\xi_m\left(\kappa(g^{-1}k)\right), e^\xi_n(k) \right\rangle dk.
\ee
The basis vectors $e^\xi_m$, $e^\xi_n \in \mathbb{H}_{\xi}$ being $K$-finite, actually belong to $C^{\infty}(K, \mathbb{H})$ and hence are bounded as functions into $\mathbb{H}$. Therefore it follows that for each $g \in G$ the integral defining $\Phi^{m,n}_{\xi, \lambda}(g)$ makes sense for $\la \in \mathfrak{a}_{\C}^*$. Moreover the function $g \mapsto \Phi^{m,n}_{\xi, \lambda}(g)$ is continuous and the function $\la \mapsto \Phi^{m,n}_{\xi, \lambda}(g)$ extends as an entire function of $\la \in \mathfrak{a}_{\C}^* \cong \C^l$. From (\ref{phimn}) and (\ref{philambda}) we get the following easy estimate 
\bes
\left|\Phi^{m,n}_{\xi, \lambda}(g)\right| \leq C \int_K e^{-(\Im(\la) + \rho)H(g^{-1}k)} dk = C \phi_{i\Im(\la)}(g), \:\:\:\: \txt{ for } g \in G.
\ees
Using the bi-$K$-invariance of $\phi_{i\Im(\la)}$ and (\ref{phiila}) we get that 
\be \label{Phi}
\left|\Phi^{m,n}_{\xi, \lambda}(g)\right| \leq C e^{\Im(\la)^+(\log a)} \phi_0(a) \leq C e^{\| \Im(\la)\|_B \|\log a\|_B } \phi_0(a), \:\:\:\: \txt{ for } g \in G,
\ee
where $g = k_1 a k_2$; $k_1, k_2 \in k$ and $a \in \overline{A_+}$. 

We have already defined $\pi_{\xi, \la}$ for all $\la \in \mathfrak{a}_{\C}^*$ and noted that $\pi_{\xi, \la}$  is a unitary operator on the Hilbert space $\mathbb{H}_{\xi}$ for $\la \in \mathfrak{a}^*$. In general for $\la \in \mathfrak{a}_{\C}^* \setminus \mathfrak{a}^*$, the representation $\pi_{\xi, \la}$ may not be unitary. However for $f \in C_c(G)$, it is easy to see that $\pi_{\xi, \la}(f)$ makes sense for all $\la \in \mathfrak{a}_{\C}^*$ and $\xi \in \widehat{M}$. We shall also need the following lemma which is essentially proved in \cite{CSS} using Harish-Chandra's subquotient theorem. 

\begin{lem} \label{lem}
Let $f \in C_c(G)$. If $\pi_{\xi, \lambda}(f) = 0$ for all $\la \in \mathfrak{a}_{\C}^*$ and $\xi \in \widehat{M}$ then $f$ is zero.
\end{lem}

Now we shall present the first main result, which is an analogue of Theorem \ref{euclidean} (a). 

\begin{thm} \label{semisimple}
Let $\psi:[0,\infty)\rightarrow[0,\infty)$ be a locally integrable function such that 
\bes 
I = \int_{0}^{\infty}{\frac{\psi(r)}{1+r^2}dr}.
\ees
\begin{enumerate}
\item[(a)] Suppose $f \in C_c(G)$ satisfies the estimate 
\be \label{decay} 
\|\pi_{\xi, \lambda}(f)\|_{HS} \leq C_{\xi}e^{-\psi(\|\lambda\|_B)}, \:\:\:\: \txt{ for } \xi \in \widehat M, \la \in \mathfrak a^* \cong \R^l,
\ee 
where $C_{\xi}$ is a positive constant depending on $\xi$. If $I=\infty$ then $f$ is zero on $G$.
\item[(b)] If $I$ is finite and $\psi$ is nondecreasing then there exists a nontrivial $f \in C_c^{\infty}(K \backslash G /K)$ satisfying the estimate (\ref{decay}).
\end{enumerate}
\end{thm}

\begin{proof}
First we shall prove (a). For fixed $\xi\in \widehat M$ and $m,n \in \mathbb{N}$ we  define 
\be \label{Fmn} 
F^{m,n}_\xi(\lambda) = \ds{\left\langle \pi_{\xi, \lambda}(f) e^\xi_m, e^\xi_n \right\rangle_\xi=  \displaystyle\int_{G}{f(g) ~ \Phi^{m,n}_{\xi,\lambda}(g)dg}}, \:\:\:\: \txt{ for } \lambda \in \mathfrak a^*.
\ee 
Since $f \in C_c(G)$ and $\Phi^{m,n}_{\xi,\lambda}(g)$ is a continuous function of $g \in G$ as well as an entire function of $\lambda \in \mathfrak a^*_{\mathbb{C}},$ it follows that $F^{m,n}_\xi$ extends as an entire function of 
$\lambda \in \mathfrak a^*_{\mathbb{C}}.$ From Lemma \ref{lem}, it is enough to 
prove that for each fixed  $m,n \in \mathbb{N}$, $F^{m,n}_\xi(\lambda)$ is zero for all $\lambda \in \mathfrak a^*_{\mathbb{C}}$. Using polar decomposition, (\ref{Fmn}) and (\ref{Phi}), it follows that 
\be \label{est}
|F^{m,n}_{\xi}(\lambda)|\leq\displaystyle\int_{K}{\displaystyle\int_{\overline{\mathfrak{a}_+}}{\displaystyle\int_{K}{|f(k_1(\exp H)k_2)|~e^{\|\Im(\lambda)\|_B\|H\|_B}~\phi_0(\exp H)~J(\exp H)~dk_1~dH~dk_2}}}.
\ee 
Since $f \in C_c(G)$, it follows from (\ref{est}), (\ref{phi0}) and (\ref{jest}) that there exists a constant $\gamma>0$ such that   
\be \label{fmnestc}
|F^{m,n}_\xi(\lambda)|\leq  C_f e^{\gamma\|\Im (\lambda)\|_B} \leq C_f e^{\gamma\|\lambda\|_B}, \:\:\:\: \txt{ for all } ~~ \lambda\in \mathfrak a^*_{\mathbb{C}} \cong \mathbb{C}^l.
\ee
Now by (\ref{decay}) and the fact that 
\bes
\left|\langle \pi_{\xi, \lambda}(f) e^\xi_m, e^\xi_n \rangle_\xi \right| \leq \|\pi_{\xi, \lambda}(f) \|_{HS}, \:\:\:\: \txt{ for } ~~ \lambda \in \mathfrak a^{*},
\ees 
we get that
\be \label{fmnestr}
|F^{m,n}_\xi(\lambda)| \leq C_{\xi} e^{-\psi(\|\lambda\|_B)}, \:\:\:\: \txt{ for all } ~~ \lambda \in \mathfrak a^{*} \cong \mathbb{R}^l.
\ee
Since all norms are equivalent on finite dimensional spaces, (\ref{fmnestc}), (\ref{fmnestr}) and Lemma \ref{lemma} implies that $F^{m,n}_\xi$ vanishes on $\mathfrak a^*_{\mathbb{C}} \cong \C^l.$ By varying $m,n$ over $\N$ and $\xi$ over $\widehat{M}$ it now follows from Lemma \ref{lem} that $f$ vanishes on $G$.

Now we shall prove (b). Since $I$ is finite and $\psi$ is nondecreasing, by Theorem \ref{euclidean} (a) there exists a nontrivial radial function $f_0 \in C_c^{\infty}(\R^l)$ satisfying (\ref{estintro}). Since $f_0$ is radial on $\R^l$, it can be thought of as a $W$-invariant function on $A \cong \R^l$. So there exists $f \in C_c^{\infty}(K \backslash G/K)$ such that $\mathcal{A}f = f_0$. Using (\ref{sphtftrel}) and (\ref{abelsphtrel}) it follows that $f \in C_c^{\infty}(K \backslash G/K)$ satisfies the estimate (\ref{decay}).

\end{proof}

\begin{rem}
If we think of a function on $X = G/K$ as a right-$K$-invariant function on $G$ then the following theorem is, in view of the relation (\ref{hftftrel}), an easy corollary of Theorem \ref{semisimple}.
\begin{thm}
Let $\psi$ and $I$ be as in Theorem \ref{semisimple}. 
\begin{enumerate}
\item[(a)] Let $f \in C_c(X)$ and the Helgason Fourier transform $\widetilde{f}$ of $f$ satisfies the estimate 
\be \label{esthft}
|\widetilde{f}(\la,kM)| \leq C e^{- \psi(\|\la\|_B)}, \:\:\:\: \txt{ for } \lambda \in \mathfrak a^{*}, ~ kM \in K/M.
\ee
If $I=\infty$ then $f=0$.
\item[(b)] If $I$ is finite and $\psi$ is nondecreasing then there exists a nontrivial $K$-invariant $f \in C_c^{\infty}(X)$ satisfying the estimate (\ref{esthft}).
\end{enumerate} 
\end{thm}
\end{rem}

\subsection{Complex Semisimple Lie Groups}
 
We shall prove an analogue of Theorem \ref{euclidean} (b) for bi-$K$-invariant functions on a noncompact, complex semisimple Lie group using the explicit expression of the elementary spherical functions $\phi_\la$ available in this case. First we shall recall the following result proved in \cite{BRS} which is needed in our proof.
 
\begin{lem}\label{openpoly}
Let $P$ be a polynomial on $\R^n$ and $\theta : \R^n \ra [0, \infty)$ be a decreasing radial function with $\lim_{\|\xi\| \ra \infty} \theta(\xi) = 0$ and 
\bes
I=\int_{\|\xi\|\geq 1}\frac{\theta (\xi)}{\|\xi\|^n}d\xi.
\ees
\begin{enumerate}
\item[(a)] Let $f\in L^1(\R^n)$ be a nontrivial function satisfying the estimate
\be \label{estlem}
|\widehat{f}(\xi)|\leq C |P(\xi)| e^{-\|\xi\|\theta(\xi)}, \:\:\:\: \txt{ for all } \xi \in \R^n.
\ee
If $f$ vanishes on a nonempty open set then $I$ is finite. 
\item[(b)] If $I$ is finite then there exists a nontrivial $f\in C_c^{\infty}(\R^n)$ satisfying (\ref{estlem}).
\end{enumerate}
\end{lem}

If $G$ is noncompact, complex semisimple, it is known that the Haar measure on $A$ corresponding to the polar decomposition is given by $\phi(H)^2 dH$ for $H \in \mathfrak{a}$ 
where $\phi(H)$ is defined by the formula
\be \label{phidefn}
\phi(H) = \sum_{s\in W}\det(s) ~ e^{s\rho(H)}, \:\:\:\: \txt{ for } H\in \mathfrak{a},
\ee
(see \cite{CN}, P. 907 and P. 910). It is also known that on a noncompact, complex semisimple Lie group, the elementary spherical functions are given by the expression
\be \label{philambdacomplex}
\phi_\lambda(H)=c(\lambda)\frac{\sum_{s\in W}{\det(s)e^{-is\lambda(H)}}}{\phi(H)}, \:\:\:\: \txt{ for } H\in \mathfrak{a}, ~ \lambda \in \mathfrak{a}^*,
\ee
and the function $\la \mapsto c(\lambda)^{-1},$ for $\la \in \mathfrak{a}^*$, is of polynomial growth (\cite{H2}, P. 432, Theorem 5.7). Thus for $f \in L^1(K \backslash G /K)$ the spherical transform $\widetilde{f}$ of $f$ is given by the integral
\be \label{sphericalft}
\widetilde f(\lambda)=\int_{\mathfrak{a}} f(H)\phi_{\lambda}(H)\phi(H)^2dH, \:\:\:\: \txt{ for } \lambda \in \mathfrak a^*.
\ee


\begin{thm} \label{complexsemisimple}
Let $G$ be a connected, noncompact, complex semisimple Lie group and $\theta:[0,\infty) \rightarrow [0,\infty)$ be a decreasing function with $\lim_{r \to\infty}\theta(r)=0$ and
\bes
I=\int_{1}^{\infty} \frac{\theta (r)}{r}dr.
\ees
\begin{enumerate}
\item[(a)] Let $f\in L^1(K \backslash G/K)$ satisfy the estimate
\be \label{sphericaldecay}
|\widetilde f(\lambda)| \leq C e^{-\|\lambda\|_B \theta(\|\lambda\|_B)}, \:\:\:\: \txt{ for } \lambda \in \mathfrak a^* \cong \R^l.
\ee
If $f$ vanishes on a nonempty open set in $G$ and $I = \infty$ then $f=0$. 
\item[(b)] If $I$ is finite then there exists a nontrivial $f\in C_c^{\infty}(K \backslash G/K)$ satisfying the estimate (\ref{sphericaldecay}).
\end{enumerate}
\end{thm}

\begin{proof}
First we shall prove (a). Using (\ref{philambdacomplex}) and (\ref{shlambda}) the spherical transform of $f$ can be written as
\beas
\widetilde f(\lambda) 
&=& c(\lambda)\sum_{s\in W} \det(s) \int_{\mathfrak{a}}{e^{-is\lambda(H)}f(H)\phi(H) dH} \\
&=& c(\lambda)\sum_{s\in W} \det(s) \int_{\mathfrak{a}}{e^{-iB(H,sH_\lambda)}f(H)\phi(H) dH}\\
&=& c(\lambda) \sum_{s\in W} \det s ~ \widehat g(sH_\lambda), \:\:\:\: \txt{ for } \lambda \in \mathfrak a^*,
\eeas
where the function $g$ on $\mathfrak a$ is defined as 
\bes
g(H)=f(H)\phi(H), \:\:\:\: \txt{ for } H \in \mathfrak a,
\ees 
and $\widehat{g}$ denotes the Euclidean Fourier transform of $g$. 
As $f \in L^1(K \backslash G/K)$ it follows that 
\bes
\int_{\mathfrak{a}} g(H)dH = \int_{\mathfrak{a}} f(H) \phi(H) dH < \infty.
\ees
Consequently $\widehat{g}$ is well defined. Since $f$ is bi-$K$-invariant and $\phi$ is odd under the action of the Weyl group, that is, 
\bes
\phi(sH) = \det(s) \phi(H), \:\:\:\: \txt{ for } H \in \mathfrak a, s \in W,
\ees
we get that
\bes
g(sH)= \det(s) g(H), \:\:\:\: \txt{ for } H \in \mathfrak a, s \in W.
\ees 
It follows that $\widehat{g}$ is also odd under the action of the Weyl group, that is,
\bes
\widehat{g}(sH_{\lambda})= \det(s) \widehat{g}(H_{\lambda}), \:\:\:\: \txt{ for } \lambda \in \mathfrak a^*, s \in W.
\ees 
Consequently the spherical transform of $f$ can be written as
\bes
\widetilde{f}(\lambda) = c(\lambda) |W| \widehat g(H_{\lambda}), \:\:\:\: \txt{ for } \lambda \in \mathfrak a^*,
\ees
where $|W|$ denotes the number of elements in the Weyl group $W$. From the estimate of $\widetilde{f}$ given in (\ref{sphericaldecay}) it follows that
\bes
|\widehat g(H_{\lambda})|\leq C |c(\lambda)|^{-1} e^{-\theta(\|\lambda\|_B)\|\lambda\|_B}, \:\:\:\: \txt{ for } \lambda \in \mathfrak a^* \cong \R^l.
\ees
Since $c(\la)^{-1}$ is of polynomial growth, using the equivalence of all norms on finite dimensional spaces we can apply Lemma \ref{openpoly} to the function $g$ on $\mathfrak{a}$ to get that $g$ is zero. Hence it follows that $f$ is zero. 

To prove (b), first we note that by Theorem \ref{euclidean} (b) we get a nontrivial radial function $f_0 \in C_c^{\infty}(\R^l)$ satisfying (\ref{estintro}). Then we proceed as in Theorem \ref{semisimple} (b) to construct $f\in C_c^{\infty}(K \backslash G/K)$ satisfying the estimate (\ref{sphericaldecay}).
\end{proof}

\section{Unique Continuation Property of Solutions to the Schr\"odinger Equation}

We consider the initial value problem for the time-dependent Schr\"odinger equation on $\R^n$ given by
\beas
\frac{\partial u}{\partial t}(x,t) ~ - ~ i\Delta u (x,t) &=& ~ 0, \:\:\:\: \txt{ for } (x,t) \in \R^n \times \R, \\
u(x,0) &=& f(x), \:\:\:\: \txt{ for } x \in \R^n.
\eeas 
Our aim is to obtain sufficient conditions on the behaviour of the solution $u$ at two different times $t = 0$ and $t = t_0$ which guarantee that $u \equiv 0$ is the unique solution of the above equation. It has recently been observed that uncertainty principles can be used to obtain such sufficient conditions. We refer the reader to \cite{EKPV} and the references therein for results in this regard. These results were further generalized in the context of noncommutative groups in \cite{C, PS, BTD, LM, BRS, BS}. In this section we wish to relate the theorems of Ingham and Paley-Wiener to the above mentioned problem in the context of symmetric spaces. We first deduce one such result for the damped Schr\"odinger Equation on $\R^n$. 

\subsection{Euclidean Space}

We consider the initial value problem for the time-dependent damped Schr\"odinger Equation on $\R^n$ given by
\begin{eqnarray}\label{dampschr}
\left\{\begin{array}{rcll} 
\displaystyle{\frac{\partial u}{\partial t}(x,t) ~ - ~ i(\Delta-c) u (x,t) } &=& ~ 0,
& \:\:\:\: \textmd{for } (x,t) \in \R^n \times \R, \\
u(x,0) &=& f(x), & \:\:\:\: \textmd{for } x \in \R^n,
\end{array}\right.
\end{eqnarray} 
where $\Delta$ denotes the Laplacian on $\R^n$, $c \in \R$ is the damping parameter and $f \in L^1(\R^n)$. It follows that the unique solution $u_t(x) = u(x,t)$ of (\ref{dampschr}) is characterized via the Euclidean Fourier transform of $u_t$ given by the equation
\be \label{characterize}
\widehat{u_t}(\xi) = e^{-it(\|\xi\|^2+c)} \widehat{f}(\xi), \:\:\:\: \textmd{ for } \xi \in \R^n.
\ee

\begin{thm} \label{dampschrthm}
Let $u$ be the solution to the equation (\ref{dampschr}) given by (\ref{characterize}), with initial value $f\in C_c(\R^n)$. Suppose there exists $t_0>0$ such that 
\be \label{estdampschr}
|u(x, t_0)| \leq C e^{-\psi(\|x\|)}, \:\:\:\: \txt{ for } x \in \R^n,
\ee
where $\psi:[0,\infty) \ra [0,\infty)$ is a locally integrable function. If 
\bes
\int_{1}^{\infty}{\frac{\psi(r)}{1+r^2}dt}=\infty,
\ees 
then $u$ is zero.
\end{thm} 

\begin{proof}
The solution $u_t(x)=u(x,t)$ of (\ref{dampschr}) is given by 
\bes
u_t(x) = \gamma_{c,t} * f(x), \:\:\:\: \txt{ for } x \in \R^n, ~ t \in \R,
\ees
where the kernel $\gamma_{c,t}$ is defined by
\be \label{gammact}
\gamma_{c,t}(x) = (4\pi|t|)^{-n/2} e^{-ict} e^{-i\pi sign(t)n/4} e^{i\|x\|^2/4t}, \:\:\:\: \txt{ for } x \in \R^n, ~ t \in \R,
\ee
so that its Fourier transform is
\bes
\widehat{\gamma_{c,t}}(\xi) = e^{-i(\|\xi\|^2+c)t}, \:\:\:\: \txt{ for } \xi \in \R^n, ~ t \in \R,
\ees
(see \cite{PS}, P. 869). Here $sign(t) = t/|t|$ denotes the sign of $t$. Using the expression for $\gamma_{c,t}$ in (\ref{gammact}) we can express the solution $u(x,t_0)$ at time $t=t_0>0$ as
\be \label{solnft}
u(x,t_0) = (4\pi t_0)^{-n/2} e^{-ict_0} e^{-i\pi n/4} e^{i\frac{\|x\|^2}{4t_0}} \widehat{h}\left(\frac{x}{2t_0}\right)
\ee
where the function $h$ on $\R^n$ is defined by
\be \label{hdefn}
h(y) = e^{i\frac{\|y\|^2}{4t_0}} f(y), \:\:\:\: \txt{ for } y \in \R^n.
\ee
Using (\ref{estdampschr}) and (\ref{solnft}) and applying Theorem \ref{euclidean} (a)
to the function $h \in C_c(\R^n)$ we get that $h$ is zero. Hence $f$ is zero and so is $u$.
\end{proof}

\begin{rem}
It is easy to see from the relations (\ref{solnft}) and (\ref{hdefn}) that the result corresponding to Theorem \ref{euclidean} (b) can also be proved similarly if we assume the corresponding conditions on the function $f$ and the solution $u(\cdot, t_0)$.  
\end{rem}

\subsection{Riemannian Symmetric Space}

Let $X=G/K$ be a Riemannian symmetric space of noncompact type where $G$ is a noncompact, connected semisimple Lie group with finite center and $K$ is a maximal compact subgroup of $G.$ We have a $G$-invariant Riemannian metric on $X$ induced by the Killing form $B$ restricted to $\mathfrak{p}$ and we can form a Laplace-Beltrami operator $\Delta$ using this metric (see \cite{H2}). We consider the initial value problem for the time-dependent Schr\"odinger equation on $X$ given by
\begin{eqnarray}\label{schrX}
\left\{\begin{array}{rcll} 
\displaystyle{\frac{\partial u}{\partial t}(x,t) ~ - ~ i\Delta u (x,t) } &=& ~ 0,
& \:\:\:\: \textmd{for } (x,t) \in X \times \R, \\
u(x,0) &=& f(x), & \:\:\:\: \textmd{for } x \in X.
\end{array}\right.
\end{eqnarray}
For $g \in G$, we define $\sigma(g) = d(gK, K)$ where $d$ is the canonical distance function on  $X = G/K$ coming from the Riemannian metric on $X$ induced by the Killing form $B$. The function $\sigma$ is bi-$K$-invariant and continuous. If $g=k_1\exp H k_2$ with $g\in G, H\in \overline{\mathfrak a_+}$ and $k_1,k_2\in K$, then 
\bes 
\sigma(g)=\sigma(\exp H)=\|H\|_B.
\ees

\subsubsection{Complex Semisimple Lie Groups}

We shall first prove a unique continuation property of solution to the Schr\"odinger equation (\ref{schrX}) in the context of complex semisimple Lie groups with bi-$K$-invariant initial data corresponding to Theorem \ref{euclidean} (b). In this case, the solution $u(\cdot, t)$ is also bi-$K$-invariant and can be considered as a function on $\mathfrak{a}$. It turns out that the situation here is slightly different from that of Euclidean spaces. The reason is roughly speaking the exponential growth of the $G$-invariant measure on $G/K$. Consequently, to prove this unique continuation property in the context of complex semisimple Lie groups we need to impose more decay on the solution $u(x,t)$. We will also show that the unique continuation property does not hold in the absence of such extra decay (see Remark \ref{remark}). Our method of proof will follow that in \cite{C} and will use some of their notation and calculations. 

\begin{thm} \label{schrcomplexthm}
Let $\theta:[0,\infty)\rightarrow [0,\infty)$ be a decreasing function with $\lim_{|r|\to\infty}\theta(r)=0$ and $u$ be a solution to the equation (\ref{schrX}) with initial value $f\in L^1(K \backslash G/K)$. Suppose there exists $t_0>0$ such that 
\be \label{estcomplex}
|u(H, t_0)|\leq C ~ \phi_0(H)~ e^{-\|H\|_B\theta(\|H\|_B)}, \:\:\:\: \txt{ for } H \in \mathfrak{a}. 
\ee
If $f$ vanishes on an open set in $\mathfrak{a} \cong \R^l$ and 
\be \label{thetaest}
I = \int_{1}^{\infty}{\frac{\theta(r)}{r}dr}=\infty
\ee 
then $u$ is identically zero.
\end{thm}

\begin{proof}
It is proved in \cite{C} that the solution $u(H,t)$ of (\ref{schrX}) at $t =t_0$ can be written as
\be \label{solncomplex}
u(H,t_0) \phi(H) = C |W|^2 t_0^{-l/2} e^{-i\left(t_0\|\rho\|_B^2 - \frac{\|H\|_B^2}{4t_0}\right)} \widehat{g_f}\left(\frac{H}{2t_0}\right), \:\:\:\: \txt{ for } H \in \mathfrak{a}, 
\ee
where the function $g$ on $\mathfrak{a}$ is defined by
\be \label{gdefn}
g_f(H) = e^{i\frac{\|H\|_B^2}{4t_0}} f(H) \phi(H), \:\:\:\: \txt{ for } H \in \mathfrak{a},
\ee
where $\phi$ is given by (\ref{phidefn}). Since $f$ vanishes on an open set in $\mathfrak{a} \cong \R^l$, $g$ also vanishes on the same open set in $\mathfrak{a}$. From the estimate of the Jacobian $J$ given in (\ref{jest}) we get that 
\be \label{phiest}
\phi(H) = \sqrt{J(\exp H)} \leq C e^{\rho(H)}, \:\:\:\: \txt{ for } H \in \mathfrak{a}. 
\ee
It follows from (\ref{estcomplex}), (\ref{solncomplex}), (\ref{phiest}) and (\ref{phi0}) that for $H \in \mathfrak{a}$ 
\be \label{uest}
\left|\widehat{g_f}\left(\frac{H}{2t_0}\right)\right| \leq C \phi_0(H) e^{-\|H\|_B\theta(\|H\|_B)} e^{\rho(H)}  \leq C (1+\|H\|_B)^m e^{-\|H\|_B\theta(\|H\|_B)}.
\ee
As $g_f$ vanishes on an open set and $\widehat{g_f}$ satisfies (\ref{uest}) it follows from Lemma \ref{openpoly} that $g_f$ is zero. Therefore $f$ is zero and hence so is $u$.
\end{proof}

\begin{rem} \label{remark}
We give an example to show that if $0 \leq \alpha <1$ then 
\bes
|u(H, t_0)|\leq C ~ \phi_0(H)^{\alpha}~ e^{-\|H\|_B\theta(\|H\|_B)}, \:\:\:\: \txt{ for } H \in \mathfrak{a} \cong \R,
\ees
with $\theta$ as in Theorem \ref{schrcomplexthm} satisfying $I = \infty$ does not imply that $u \equiv 0$. We consider the group $G = SL(2,\C)$ (see \cite{H2}, P. 433; \cite{Te}, P. 313). In this case
\bes 
A=\left\{\left ( \begin{matrix} e^H & 0 \\ 0 & e^{-H} \\  \end{matrix} \right): H \in\R\right\} \txt{ and } \mathfrak{a}=\left\{A_H = \left (\begin{matrix} H & 0 \\ 0 & -H \\ \end{matrix} \right): H \in\R\right\}.
\ees
It is known that the Jacobian of the Haar measure is $\sinh ^2 2H$ so that we have 
\bes 
\phi(H) = \sinh 2H, \:\:\:\: \txt{ for } H \in \R,
\ees 
and the distance function is given by 
\bes
\|H\|_B = 4|H|, \:\:\:\: \txt{ for } H \in \R.
\ees 
The elementary spherical functions are given by  
\beas
\phi_\lambda(H) &=& C \frac{\sin \la H}{ \la \sinh 2H}, \:\:\:\: \txt{ for } \la, ~ H \in \R, \\
\phi_0(H) &=& C\frac{H}{\sinh 2H}, \:\:\:\: \txt{ for } H \in \R.
\eeas
We fix $0 \leq \alpha< 1$ and $t_0>0$. We will give an example of a non-zero initial value $f \in C_c(K \backslash G /K)$ such that the corresponding solution $u$ to the Schr\"odinger equation (\ref{schrX}) satisfies 
\be \label{exampleest}
|u(H, t_0)|\leq C ~ \phi_0(H)^{\alpha}~ e^{-\|H\|_B\theta(\|H\|_B)}, \:\:\:\: \txt{ for } H \in \mathfrak{a} \cong \R,
\ee
where $\theta$ is as in Theorem \ref{schrcomplexthm} satisfying (\ref{thetaest}).
We consider $\eta \in \R$ such that $0< \eta< 1-\alpha$ and define $\beta = 1-\alpha-\eta.$ For any $h \in C_c^{\infty}(\R)$ we have 
\be \label{betaest}
|\widehat{h}(H)| \leq C e^{\beta|H|}, \:\:\:\: \txt{ for all } H \in \R,
\ee 
We consider such a non-zero function $h$ which is supported in $[\beta', \beta] \subset [-\beta, \beta] \subset \R$ where $0 < \beta' < \beta$. Since $\theta$ is a function decreasing to zero, there exists $M_1>1$ such that  
\be \label{thetabound}
\theta(4|H|) < \frac{\eta}{4}, \:\:\:\: \txt{ for } |H| > M_1, 
\ee
and there exists $M_2>1$ such that 
\be \label{sinhest}
e^{|H|} \leq C \sinh2|H|, \:\:\:\: \txt{ for } |H| > M_2.
\ee 
From (\ref{betaest}), (\ref{thetabound}) and (\ref{sinhest}) it follows that for $M = \max\{M_1,M_2\} >1 $ and $|H|>M$
\bea \label{example}
|\widehat h(H)| 
&\leq& C e^{(1-\alpha)|H|}e^{-\eta|H|} \nonumber \\
&\leq& C |H|^{\alpha}(\sinh2|H|)^{1-\alpha} e^{-4|H|\theta(4|H|)} \nonumber \\
&=& C \left(\frac{|H|}{\sinh2|H|}\right)^{\alpha}(\sinh2|H|) e^{-4|H|\theta(4|H|)}  \nonumber \\
&=& C|\phi(H)| |\phi_0(H)|^{\alpha} e^{-4|H|\theta(4|H|)}.
\eea
We define a function $f \in C_c^{\infty}(\R)$ by 
\be \label{fdefn}
f(H)= \frac{1}{2t_0} e^{-4i\frac{|H|^2}{t_0}}\phi(H)^{-1}h\left(\frac{H}{2t_0}\right), \:\:\:\: \txt{ for } H\in \R.
\ee 
This is well defined since $\phi$ vanishes only at $0$ and $h$ is supported away from zero. We consider the solution $u(H,t)$ of the system (\ref{schrX}) with the initial value $f$. It follows from (\ref{gdefn}) and (\ref{fdefn}) that 
\bes
g_f (H) = \frac{1}{2t_0} h\left(\frac{H}{2t_0}\right), \:\:\:\: \txt{ for } H\in \R. 
\ees
So by (\ref{solncomplex}) 
\bes
|u(H,t_0)||\phi(H)| = C |\widehat{h}(H)|, \:\:\:\: \txt{ for } H \in \R.
\ees
Hence by (\ref{example}) we get that 
\bes
|u(H,t_0)||\phi(H)| \leq C |\phi(H)| |\phi_0(H)|^{\alpha} e^{-4|H|\theta(4|H|)}, \:\:\:\: \txt{ for } |H| > M.
\ees
Since by continuity the function 
\bes
H \mapsto u(H,t_0) e^{4|H|\theta(4|H|)} (\phi_0(H))^{-\alpha} 
\ees
is bounded on the compact set $[0,M]$ it follows that the solution $u$ satisfies (\ref{exampleest}) with non-zero initial data $f \in C_c^{\infty}(\R)$.

\end{rem}

\subsubsection{Symmetric Space of Noncompact Type}

We shall now prove a unique continuation property of solution to the Schr\"odinger equation (\ref{schrX}) in the context of Riemannian symmetric spaces of noncompact type using Theorem \ref{dampschrthm}. As in the case of complex semisimple Lie groups, here also we need to impose extra decay on the solution $u(x,t)$ and using a similar example we will show that such decay is necessary (see Remark \ref{remarksym}). However, here we shall prove our result by reducing the problem to the Euclidean case via the Radon transform. It can be seen that for the initial data $f \in L^2(X)$ there exists unique solution $u(\cdot, t) \in L^2(X)$ satisfying the Schr\"odinger equation (\ref{schrX}) in the sense of distributions. However the Radon transform is not defined in general for functions in $L^2(X)$. So we need to consider our initial data $f$ in the $L^2$-Schwartz space $\mathcal{S}(X)$ consisting of smooth rapidly decreasing functions on $X$, on which the Radon transform happens to be defined. We shall start with the definition of $\mathcal{S}(X)$.

Let $\mathcal{D}_L(G)$ and $\mathcal{D}_R(G)$ denote the algebra of left invariant and that of right invariant differential operators on $G$ respectively. The $L^2$-Schwartz space $\mathcal{S}(G)$ is defined as the space of smooth functions $f$ on $G$ such that for each $D \in \mathcal{D}_L(G)$, $E \in \mathcal{D}_R(G)$ and $l \in \N$
\be \label{schwartzsp}
\sup_{g \in G} |(1+\sigma(g))^l \phi_0(g)^{-1} (DEf)(g)| < \infty.
\ee 
The $L^2$-Schwartz space $\mathcal{S}(X)$ is then defined as the space of $f \in \mathcal{S}(G)$ which are right invariant under $K$ (see \cite{H1}, P. 214). 

For $\lambda\in \mathfrak{a}_{\C}^*$ and $kM\in K/M$ we define the function $e_{\lambda, k}:X \rightarrow \C$ by 
\bes
e_{\lambda, k}(x)= e^{(i\lambda-\rho)H (x^{-1}k)}, \:\:\:\: \txt{ for }  x = gK \in X, g \in G.
\ees
It is known that these functions $e_{\lambda, k}$ appearing in the definition of the Helgason-Fourier transform (\ref{hftdefn}) are eigenfunctions of the Laplace-Beltrami operator $\Delta$ given by
\bes
\Delta e_{\lambda, k}= -(\|\lambda\|_B^2+\|\rho\|_B^2)e_{\lambda, k}, \:\:\:\: \txt{ for } \lambda \in \mathfrak{a}_{\C}^*, ~ kM\in K/M,
\ees
(see \cite{H1}, P. 99 and \cite{H2}, Ch. II). It follows that for $f \in L^2(X)$ the unique solution $u(\cdot, t) \in L^2(X)$ satisfying the Schr\"odinger equation (\ref{schrX}) is characterized by 
\be \label{hftsoln}
\widetilde{u_t}(\lambda, kM)= e^{-it\left(\|\lambda\|_B^2+\|\rho\|_B^2\right)}\widetilde{f}(\lambda, kM), \:\:\:\: \txt{ for } kM \in K/M \txt{ and } \lambda \in \mathfrak{a}^*.
\ee

However we are interested in the special case where the initial data $f \in \mathcal{S}(X)$. In this case it turns out that the solution $u_t(\cdot) = u(t,\cdot) \in \mathcal{S}(X)$ (see \cite{PS}, P. 872; \cite{E}, Theorem 4.1.1). It follows that that the analysis of solutions of the Schr\"odinger equation on $\mathbb{R}^n$ carries out to $X$ when the Fourier transform on $\mathbb{R}^n$ is replaced by the Helgason-Fourier transform. For any $f \in \mathcal{S}(X)$ the Radon transform $Rf$ is well defined (see \cite{H1}, P. 218, Theorem 1.15). From (\ref{hftradon}) and (\ref{hftsoln}) it follows that for $\lambda \in \mathfrak{a}^*$ and fixed $kM \in K/M$ we have
\be \label{radonsoln}
\mathcal{F}_A((Ru_t)(kM,\cdot)(\lambda)=e^{-it\left(\|\lambda\|_B^2+\|\rho\|_B^2\right)}\mathcal{F}_A((Rf)(kM,\cdot)(\lambda).
\ee

We shall recall few facts which will be needed. Firstly, the function $\sigma$ satisfies the following inequality
\be\label{sigmaineq}
\sigma(an) \geq \sigma(a), \:\:\:\: \txt{ for all } a\in A, ~ n\in N 
\ee
(see \cite{GV}, Lemma 6.2.7). Moreover, there exists a non-negative integer $m \geq 0$ such that for some constant $C_0>0$,
\be \label{unifest}
e^{\rho \log a}\int_{N}{\phi_0(an)~(1+\sigma(an))^{-m}dn} \leq C_0, \:\:\:\: \txt{ for all } a \in A,
\ee
(see \cite{GV}, P. 264, Theorem 6.2.3). We shall now present the unique continuation property for solutions to the Schr\"odinger equation (\ref{schrX}) for symmetric spaces.

\begin{thm} \label{schrthm}
Let $u(x,t)$ be a solution to the equation (\ref{schrX}) with initial value $f\in C_c^{\infty}(X)$. Suppose there exists $t_0>0$ such that 
\be \label{hypo}
|u(x, t_0)| \leq C ~ \phi_0(x)~e^{-\psi(\sigma(x))}, \:\:\:\: \txt{ for } x \in X,
\ee
where $\psi:[0,\infty) \ra [0,\infty)$ is a non-decreasing function. If 
\bes
I = \int_{1}^{\infty}{\frac{\psi(r)}{1+r^2}dr}=\infty,
\ees 
then $f$ is zero on $X$.
\end{thm} \label{SchrSM}

\begin{proof}
Since $f \in C_c^{\infty}(X) \subset \mathcal{S}(X)$, it follows that $u_t \in \mathcal{S}(X)$. Viewing $u_t$ as a right $K$-invariant function on $G$, from the definition of $L^2$-Schwartz space (\ref{schwartzsp}) we get that for each $l \in \N$
\be \label{polydecay}
|u_t(g)| \leq C \phi_0(g) (1+\sigma(g))^{-2l}, \:\:\:\: \txt{ for } g \in G.
\ee
By multiplying the inequalities in (\ref{hypo}) and (\ref{polydecay}) it follows that for each $l \in \N$
\be \label{finaldecay}
|u(x, t_0)| \leq C ~ \phi_0(x)~(1+\sigma(x))^{-l}e^{-\frac{1}{2}\psi(\sigma(x))}, \:\:\:\: \txt{ for } x \in X.
\ee
Now, from the expression of $u_t$ given in (\ref{characterize}) and the relation (\ref{radonsoln}) it follows that for fixed $kM \in K/M$, $(Ru_t)(kM,\cdot)$ is a solution to the system (\ref{dampschr}) with initial value $(Rf)(kM,\cdot)$ and damping parameter $\|\rho\|_B^2$. We wish to apply Theorem \ref{dampschrthm} to this particular solution. Let us fix $kM \in K/M$. Since $f \in C_c^{\infty}(X)$ it is easy to see that $Rf(kM,\cdot) \in C_c^{\infty}(A)$. We need to prove that $(Ru_{t})(kM,\cdot)$ satisfies the estimate (\ref{estdampschr}) at time $t=t_0$. Indeed, for $a \in A$ and $l = m$ using (\ref{radon}) and (\ref{finaldecay}) we get that  
\beas
\left|(Ru_{t_0})(kM,a)\right| &=& e^{\rho \log a}\int_{N}|{u_{t_0}(kan\cdot o)|~dn} \nonumber\\
&\leq& C~e^{\rho \log a}\int_{N}{\phi_0(an)~(1+\sigma(an))^{-m}e^{-\frac{1}{2}\psi(\sigma(an))}~dn}.
\eeas
Since $\psi$ is non-decreasing using (\ref{sigmaineq}) it follows that 
\be \label{psiest}
\psi(\sigma(an))\geq \psi(\sigma(a)), \:\:\:\: \txt{ for all } a\in A, ~ n\in N.
\ee
Using (\ref{psiest}) and (\ref{unifest}) we get that
\bes
|(Ru_{t_0})(kM,a)|\leq Ce^{-\frac{1}{2}\psi(\sigma (a))}, \:\:\:\: \txt{ for all } a\in A.
\ees
Using the equivalence of norms in finite dimensional spaces, we can invoke Theorem \ref{dampschrthm} to get that $(Rf)(kM,\cdot)$ is zero on $A$. Since this is true for all $kM \in K/M$ we get that $Rf$ is zero on $K/M \times A$. By the injectivity of Radon transform (\cite{H1}, P. 220, Corollary 1.6) we conclude that  $f$ is zero. Hence so is $u$.
\end{proof}

\begin{rem} 
\begin{enumerate} 
\item \label{remarksym}
We give an example to show that if $0 \leq \alpha <1$ then there exists a non-decreasing function $\psi$ (depending on $\alpha$) and a non-zero initial value $f \in C_c(K \backslash G /K)$ such that the corresponding solution $u$ to the Schr\"odinger equation (\ref{schrX}) satisfies
\be \label{exampledecay}
|u(H, t_0)|\leq C ~ \phi_0(H)^{\alpha}~ e^{-\psi(\|H\|_B)}, \:\:\:\: \txt{ for } H \in \mathfrak{a} \cong \R,
\ee
where $\psi$ is as in Theorem \ref{schrthm} with $I = \infty$. We again consider the group $G = SL(2,\C)$ and as in Remark \ref{remark} we get $h \in C_c^{\infty}(\R)$ supported in $[\beta', \beta] \subset [-\beta, \beta]$ with $0 < \beta' < \beta$ such that
\bes
|\widehat{h}(H)| \leq C e^{\beta|H|}, \:\:\:\: \txt{ for all } H \in \R,
\ees
where $\beta = 1-\alpha-\eta$ and $0< \eta< 1-\alpha$. We now define
\be \label{psidefn}
\psi(H)= \eta |H|, \:\:\:\: \txt{ for } H \in \R.
\ee
Using (\ref{sinhest}) and (\ref{psidefn}) in (\ref{example}) we get that for $|H| > M_2$
\bes
|\widehat h(H)| \leq C|\phi(H)| |\phi_0(H)|^{\alpha} e^{-\psi(H)}.
\ees
Defining $f \in C_c^{\infty}(\R)$ by (\ref{fdefn}) we conclude as in Remark \ref{remark}
that $u$ satisfies (\ref{exampledecay}) with non-zero initial data $f \in C_c^{\infty}(\R)$.

\item It seems to be an interesting problem to see whether Theorem \ref{complexsemisimple} and Theorem \ref{schrcomplexthm} can be generalized to general Riemannian symmetric spaces of noncompact type.

\end{enumerate}
\end{rem}

\textbf{Acknowledgement.} We would like to thank Swagato K. Ray for suggesting this problem and for the many useful discussions during the course of this work. We also thank Rudra P Sarkar for his valuable comments regarding this work.

\end{document}